\documentclass[12pt]{article}
\usepackage{enumerate}
\usepackage{amssymb,a4wide,latexsym,makeidx,epsfig,fleqn}
\usepackage{amsthm}
\usepackage{amsmath}
\usepackage{enumerate}
\usepackage{float}
\newtheorem{theorem}{Theorem}[section]

\newtheorem{lemma}[theorem]{Lemma}

\begin{document}
\textwidth 150mm \textheight 225mm
\title{The minimum Q-index of strongly connected bipartite digraphs
with complete bipartite subdigraphs
\thanks{ Supported by the National Natural Science
Foundation of China (No. 11171273) and
the Seed Foundation of Innovation and Creation for Graduate Students in Northwestern Polytechnical University (No. Z2017190).}}
\author{{Weige Xi and Ligong Wang\footnote{Corresponding author.} }\\
{\small Department of Applied Mathematics, School of Science,}\\
{\small Northwestern Polytechnical University, Xi'an, Shaanxi 710072, P.R.China}
 \\{\small E-mail: xiyanxwg@163.com, lgwangmath@163.com }\\}
\date{}
\maketitle
\begin{center}
\begin{minipage}{120mm}
\vskip 0.3cm
\begin{center}
{\small {\bf Abstract}}
\end{center}
{\small Let $\mathcal{G}_{n,p,q}$ denote the set of strongly connected
bipartite digraphs on $n$ vertices
which contain a complete bipartite subdigraph $\overleftrightarrow{K_{p,q}}$, where $p, q, n$ are positive
 integers and $p+q \leq n$. In this paper, we study the Q-index (i.e. the signless Laplacian spectral radius)
 of $\mathcal{G}_{n,p,q}$, and determine the extremal digraph that has the minimum Q-index.

\vskip 0.1in \noindent {\bf Key Words}: \ Strongly connected
bipartite digraphs, Q-index, eigenvector. \vskip
0.1in \noindent {\bf AMS Subject Classification (2000)}: \ 05C50, 15A18.}
\end{minipage}
\end{center}

\section{Introduction}
\label{sec:ch6-introduction}

Let $G=(V(G),E(G))$ be a digraph with vertex set
$V(G)=\{v_1,v_2,\ldots,v_n\}$ and  arc set
$E(G)$. If there is an arc from $v_i$ to $v_j$, we indicate this by writing
$(v_i,v_j)$, call $v_j$ the head of $(v_i,v_j)$, and $v_i$ the tail
of $(v_i,v_j)$, respectively. Let $H=(V(H),E(H))$ be a subdigraph of $G$, if
$V(H)\subseteq V(G)$, $E(H)\subseteq E(G)$. For any vertex $v_i\in V(G)$, $N_i^{+}=N_{v_i}^{+}(G)=\{v_j : (v_i,v_j)\in E(G)\}$
is called the set of out-neighbors of $v_i$. Let $d_i^{+}=|N_i^{+}|$ denote the outdegree of the vertex $v_i$ in the
digraph $G$. A digraph is simple if it has
no loops and multiarcs. A digraph is strongly
connected if for every pair of vertices $v_i,v_j\in V(G)$, there
exists a directed path from $v_i$ to $v_j$. Let $\overrightarrow{P_{n}}$ and
$\overrightarrow{C_{n}}$ denote the directed path and the
directed cycle on $n$ vertices, respectively. Suppose
$\overrightarrow{P_k}=v_1v_2\dots v_k$, we call $v_1$ the
initial vertex of the directed path
$\overrightarrow{P_k}$, $v_k$ the terminal vertex of the
directed path $\overrightarrow{P_k}$, respectively.

Let $G=(V(G),E(G))$ be a digraph, If $V(G)=U\cup W$, $U\cap W=\emptyset$
and for any arc $(v_i,v_j)\in E(G)$, $v_i\in U$ and $v_j\in W$ or $v_i\in W$ and $v_j\in U$,
then the digraph $G=(V(G),E(G))$ is called a bipartite digraph. Let
$\overleftrightarrow{K_{p,q}}$ be a complete bipartite digraph
obtained from a complete bipartite graph $K_{p,q}$ by replacing each edge with a
pair of oppositely directed arcs. In this paper, we consider finite, simple strongly connected bipartite digraphs.

For a digraph $G$ of order $n$, let $A(G)=(a_{ij})$ denote the adjacency matrix of $G$,
where $a_{ij}=1$ if $(v_i,v_j)\in E(G)$ and $a_{ij}=0$ otherwise.
Let $D(G)=\textrm{diag}(d_1^{+},d_2^{+},$ $\ldots,d_n^{+})$
be the diagonal matrix with outdegrees of the vertices of $G$ and
$Q(G)=D(G)+A(G)$ the signless Laplacian matrix of $G$. The
spectral radius of $Q(G)$, i.e., the largest modulus of the eigenvalues of $Q(G)$,
is called the Q-index or the signless Laplacian spectral radius of $G$, denoted by $q(G)$.
It follows from Perron Frobenius Theorem that $q(G)$ is an eigenvalue of $Q(G)$, and
there is a positive unit eigenvector corresponding to $q(G)$ when $G$ is a
strongly connected digraph, we called the positive unit eigenvector corresponding to $q(G)$ the
Perron vector of digraph $G$.

The problem of determining graphs that maximize or minimize the spectral radius or the largest eigenvalue
 of the related matrix among a class of given graphs is an important classic problem in spectral graph theory.
The problem determining undirected graphs that maximize or minimize of adjacent spectral radius, Laplacian spectral radius
and signless Laplacian spectral radius have been well treated in the literature, see \cite{CT,LL,PS,WH} and so on, but
there is no much known about digraphs. In \cite{LSWY}, Lin et al. characterized the extremal digraphs with minimum
spectral radius among all digraphs with given clique number and grith, and the extremal
digraphs with maximum spectral radius among all digraphs with given vertex connectivity.
In \cite{HoYo}, Hong and You established some sharp upper or lower
bound on the signless Laplacian spectral radius of digraphs with some given
parameter such as clique number, girth or vertex connectivity,
and characterized the extremal digraph. In \cite{DL}, Drury and
Lin determined the digraphs that have the minimum and second minimum spectral radius
among all strongly connected digraphs with given order and dichromatic number.
In \cite{CCL}, Chen et.al studied the spectral radius of strongly
connected bipartite digraphs which contain a complete bipartite subdigraph and given
characterization of the extremal digraph with the least spectral radius.

In this paper, we study the Q-index (i.e. the signless Laplacian spectral radius) of strongly connected
bipartite digraphs which have the complete bipartite subdigraph
$\overleftrightarrow{K_{p,q}}$ \ $(p\geq q\geq1)$, and determine the extremal digraph that has the minimum Q-index.

\section{Preliminaries}

In this section, we present some known lemmas which are useful for the proof of the main results.
In the rest of this paper, let $X=(x_1,x_2,\ldots,x_n)^T$ be the
Perron vector corresponding to $q(G)$, where $x_i$ corresponds to the vertex $v_i$.

\noindent\begin{lemma}\label{le:c1} (\cite{HoYo})
Let $G=(V(G),E(G))$ be a simple digraph on $n$ vertices, $u, v, w$ distinct vertices of
$V(G)$ and $(u,v)\in E(G)$. Let $H=G-\{(u,v)\}+\{(u,w)\}$. (Noting that if
$(u,w)\in E(G)$, then $H$ has multiple arc
$(u,w)$.) If $x_w\geq x_v$, then $q(H)\geq q(G)$.
Furthermore, if $H$ is strongly connected and $x_w>x_v$, then $q(H)>
q(G)$.
\end{lemma}

\noindent\begin{lemma}\label{le:c2} (\cite{HoYo})
Let $G$ be a digraph and $G_1,G_2,\ldots,G_s$ be the strongly connected components of $G$.
Then $q(G)=\max\{q(G_1),q(G_2),\ldots,q(G_s)\}$.
\end{lemma}

\noindent\begin{lemma}\label{le:c3} (\cite{HoYo})
Let $G$ be a digraph
and $H$ be a subdigraph of $G$. Then $q(H)\leq q(G)$. If $G$ is strongly
connected, and $H$ is a proper subdigraph of $G$, then $q(H)<q(G)$.
\end{lemma}

\noindent\begin{lemma}\label{le:c4} (\cite{HoYo})
Let $G \ (\neq \overrightarrow{C_{n}})$ be a strongly connected digraph with
vertex set $V(G)=\{v_1,v_2,\ldots,v_n\}$, $\overrightarrow{P_k}=v_1v_2\dots v_k \ (k\geq3)$
be a directed path of $G$ with $d_i^{+}=1 \ (i=2,3,\ldots,k-1)$. Then we have $x_2<x_3<\ldots<x_{k-1}<x_k$.
\end{lemma}

Let $G=(V(G),E(G))$ be a digraph with $(u,v)\in E(G)$ and $w\notin V(G)$, $G^w=(V(G^w),E(G^w))$
with $V(G^w)=V(G)\cup\{w\}, \ E(G^w)=E(G)-\{(u,v)\}+\{(u,w),(w,v)\}$.

\noindent\begin{lemma}\label{le:c5} (\cite{HoYo})
Let $G \ (\neq \overrightarrow{C_{n}})$ be a strongly connected digraph,
$w\notin V(G)$, and $G^w$ defined as before. Then $q(G)\geq q(G^w)$.
\end{lemma}

\section{ The Q-index of strongly connected
bipartite digraphs which contain a complete bipartite subdigraph}
\label{sec:ch-sufficient}

In this section, we will show that if $n\equiv p+q \ (mod \ 2)$ then $B^5_{n,p,q}$
 is the unique bipartite digraph with the minimum Q-index among all strongly connected
bipartite digraphs on $n$ verties which have the complete bipartite subdigraph $\overleftrightarrow{K_{p,q}}$,
otherwise, if $n\not\equiv p+q \ (mod \ 2)$ then $B^1_{n,p,q}$ is the unique bipartite digraph with
the minimum Q-index among all strongly connected
bipartite digraphs on $n$ verties which have the complete bipartite subdigraph $\overleftrightarrow{K_{p,q}} \ (p\geq q\geq1)$.

Let $\overleftrightarrow{K_{p,q}}$ be a complete bipartite digraph with
$V(\overleftrightarrow{K_{p,q}})=V_p \ \cup \ V_q$,
$E(\overleftrightarrow{K_{p,q}})=\{(u,v),$ $(v,u)| u\in V_p, v\in V_q\}$ and $|V_p|=p$, $|V_q|=q$.
Let $\mathcal{G}_{n,p,q}$ denote the set of strongly connected
bipartite digraphs on $n$ vertices which contain a complete bipartite subdigraph $\overleftrightarrow{K_{p,q}}$.
As we all know, if $p+q=n$, then $\mathcal{G}_{n,p,q}=\{\overleftrightarrow{K_{p,q}}\}$,
and $q(\overleftrightarrow{K_{p,q}})=p+q$. Thus we only consider the cases when $p+q\leq n-1$ and $p\geq q\geq1$.
In the rest of this section, we just discuss under this assumption.

Let $B^1_{n,p,q}=(V(B^1_{n,p,q}),E(B^1_{n,p,q}))$ be a digraph obtained by adding a directed path
$\overrightarrow{P_{n-p-q+2}}=v_1v_{p+q+1}v_{p+q+2}\ldots v_{n}v_{p}$ to a complete bipartite digraph
$\overleftrightarrow{K_{p,q}}$ such that $ $ $V(\overleftrightarrow{K_{p,q}}) \cap V(\overrightarrow{P_{n-p-q+2}})=\{v_1,v_p\}$
as shown in Figure \ref{Fig.1}(a), where $V(B^1_{n,p,q})=\{v_1,v_2,\ldots,v_n\}$. Clearly, if $n-p-q$ is odd,
then $B^1_{n,p,q}\in \mathcal{G}_{n,p,q}$.

Let $B^2_{n,p,q}=(V(B^2_{n,p,q}),E(B^2_{n,p,q}))$ be a digraph obtained by adding a directed path
$\overrightarrow{P_{n-p-q+2}}=v_{p+1}v_{p+q+1}v_{p+q+2}\ldots v_{n}v_{p+q}$ to a complete bipartite digraph
$\overleftrightarrow{K_{p,q}}$ such that$ $ $V(\overleftrightarrow{K_{p,q}}) \cap V(\overrightarrow{P_{n-p-q+2}})=\{v_{p+1},v_{p+q}\}$
as shown in Figure \ref{Fig.1}(b), where $V(B^2_{n,p,q})=\{v_1,v_2,$ $\ldots,v_n\}$. Clearly, if $n-p-q$ is odd,
then $B^2_{n,p,q}\in \mathcal{G}_{n,p,q}$.

Let $B^5_{n,p,q}=(V(B^5_{n,p,q}),E(B^5_{n,p,q}))$ be a digraph obtained by adding a directed path
$\overrightarrow{P_{n-p-q+2}}=v_1v_{p+q+1}v_{p+q+2}\ldots v_{n}v_{p+1}$ to a complete bipartite digraph
$\overleftrightarrow{K_{p,q}}$ such that $ $ $V(\overleftrightarrow{K_{p,q}}) \cap V(\overrightarrow{P_{n-p-q+2}})=\{v_1,v_{p+1}\}$
as shown in Figure \ref{Fig.2}(a), where $V(B^5_{n,p,q})=\{v_1,v_2,\ldots,$ $v_n\}$. Clearly, if $n-p-q$ is even, then $B^1_{n,p,q}\in \mathcal{G}_{n,p,q}$.

Let $B^6_{n,p,q}=(V(B^6_{n,p,q}),E(B^6_{n,p,q}))$ be a digraph obtained by adding a directed path
$\overrightarrow{P_{n-p-q+2}}=v_{p+1}v_{p+q+1}v_{p+q+2}\ldots v_{n}v_1$ to a complete bipartite digraph
$\overleftrightarrow{K_{p,q}}$ such that $ $ $V(\overleftrightarrow{K_{p,q}}) \cap V(\overrightarrow{P_{n-p-q+2}})=\{v_1,v_{p+1}\}$
as shown in Figure \ref{Fig.2}(b), where $V(B^6_{n,p,q})=\{v_1,v_2,\ldots,$ $v_n\}$. Clearly, if $n-p-q$ is even, then $B^6_{n,p,q}\in \mathcal{G}_{n,p,q}$.
\begin{figure}[H]
\centering
\includegraphics[scale=0.6]{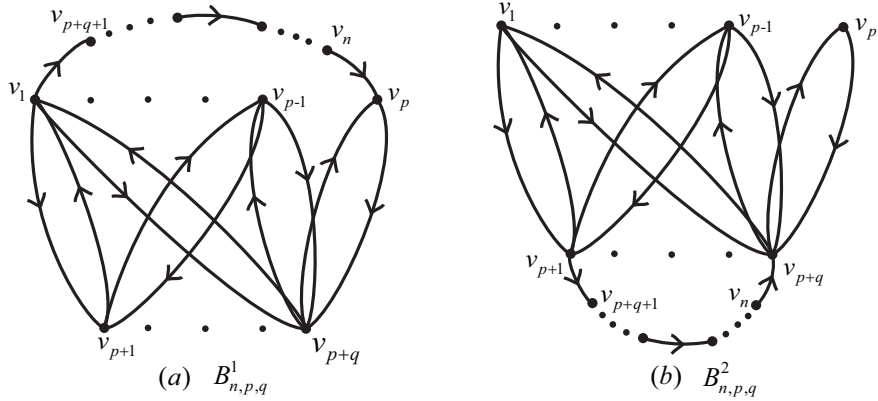}
\caption{$B^1_{n,p,q}$ and $B^2_{n,p,q}$}\label{Fig.1}
\end{figure}
\begin{figure}[H]
\centering
\includegraphics[scale=0.6]{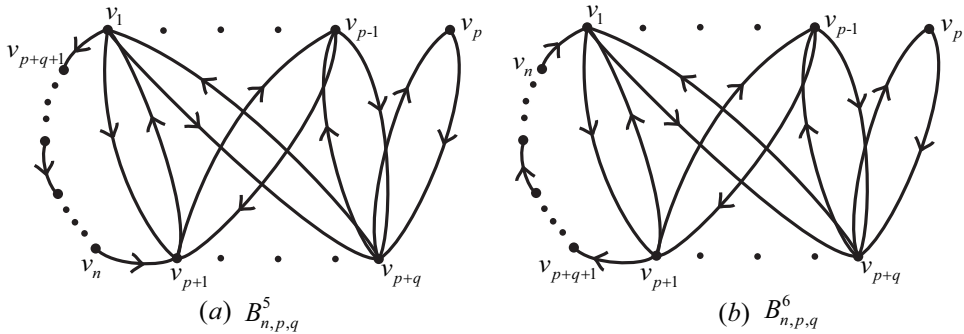}
\caption{$B^5_{n,p,q}$ and $B^6_{n,p,q}$}\label{Fig.2}
\end{figure}

\noindent\begin{theorem}\label{th:ch-1} For digraphs $B^1_{n,p,q}$ and $B^2_{n,p,q}$, as shown in Figure \ref{Fig.1},
$$q(B^1_{n,p,q})\leq q(B^2_{n,p,q}),$$ with equality if and only if $p=q$.
\end{theorem}

\begin{proof} If $p=q$, then $B^2_{n,p,q}\cong B^1_{n,p,q}$. Hence $q(B^1_{n,p,q})=(B^2_{n,p,q})$. Otherwise $p>q$, let $X=(x_1,x_2,\ldots,x_n)^T$ be the
Perron vector corresponding to $q(B^1_{n,p,q})$ where $x_i$ corresponds to the vertex $v_i$. Since $$(Q(B^1_{n,p,q})X)_i=q(B^1_{n,p,q})x_i=qx_i+\sum\limits_{u\in N^+_i}x_u=qx_i+\sum\limits_{u\in V_q}x_u, \ (i=2,3,\ldots,p),$$
then $(q(B^1_{n,p,q})-q)x_i=\sum\limits_{u\in V_q}x_u$. Noting that
$\overleftrightarrow{K_{p,q}}$ is a proper subdigraph of $B^1_{n,p,q}$ and $ $
$B^1_{n,p,q}$ is strongly connected, then by Lemma \ref{le:c3}, we have $q(B^1_{n,p,q})>q(\overleftrightarrow{K_{p,q}})=p+q$.
Thus
\begin{equation}\label{eq:1}
x_2=x_3=\ldots=x_p\triangleq x_p.
\end{equation}
 Since $$(Q(B^1_{n,p,q})X)_j=q(B^1_{n,p,q})x_j
 =px_j+\sum\limits_{u\in N^+_j}x_u=px_j+\sum\limits_{u\in V_p}x_u, \ (j=p+1,p+2,\ldots,p+q),$$
similarly, then we have
\begin{equation}\label{eq:2}
x_{p+1}=x_{p+2}=\ldots=x_{p+q}\triangleq x_{p+1}.
\end{equation}
From $Q(B^1_{n,p,q})X=q(B^1_{n,p,q})X$, (\ref{eq:1}) and (\ref{eq:2}), we have
\begin{equation}\label{eq:3}
q(B^1_{n,p,q})x_1=(q+1)x_1+x_{p+q+1}+qx_{p+1},
\end{equation}
\begin{equation}\label{eq:4}
q(B^1_{n,p,q})x_p=qx_p+qx_{p+1},
\end{equation}
\begin{equation}\label{eq:5}
q(B^1_{n,p,q})x_{p+1}=px_{p+1}+x_1+(p-1)x_p,
\end{equation}
\begin{equation}\label{eq:6}
(q(B^1_{n,p,q})-1)^{n-p-q}x_{p+q+1}=x_p.
\end{equation}
From (\ref{eq:4}), we have
\begin{equation}\label{eq:7}
(q(B^1_{n,p,q})-q)x_p=qx_{p+1}.
\end{equation}
multiply both sides of the equation $(\ref{eq:7})$ by $(q(B^1_{n,p,q})-p)$, we have
\begin{eqnarray}\label{eq:8}
(q(B^1_{n,p,q})-q)(q(B^1_{n,p,q})-p)x_p&=&(q(B^1_{n,p,q})-p)qx_{p+1} \nonumber \\
&=&qx_1+q(p-1)x_p, \ \  \ \ ( \ \textrm {by \ (\ref{eq:5})})
\end{eqnarray}
multiply both sides of the equation $(\ref{eq:8})$ by $(q(B^1_{n,p,q})-q-1)$,
we have
\begin{align*}
&(q(B^1_{n,p,q})-q)(q(B^1_{n,p,q})-p)(q(B^1_{n,p,q})-q-1)x_p \\
&=q(q(B^1_{n,p,q})-q-1)x_1+q(p-1)(q(B^1_{n,p,q})-q-1)x_p \\
&=qx_{p+q+1}+q^2x_{p+1}+q(p-1)(q(B^1_{n,p,q})-q-1)x_p   \ \ ( \ \textrm {by (\ref{eq:3})}) \\
&=qx_{p+q+1}+q(q(B^1_{n,p,q})-q)x_p+q(p-1)(q(B^1_{n,p,q})-q-1)x_p  \ \ ( \ \textrm {by (\ref{eq:7})}) \\
&=q\frac{1}{(q(B^1_{n,p,q})-1)^{n-p-q}}x_p+q(q(B^1_{n,p,q})-q)x_p+q(p-1)(q(B^1_{n,p,q})-q-1)x_p, \ \ ( \ \textrm {by (\ref{eq:8})})
\end{align*}
multiply both sides of the above equation by $(q(B^1_{n,p,q})-1)^{n-p-q}$, we have
\begin{align*}
&(q(B^1_{n,p,q})-1)^{n-p-q}(q(B^1_{n,p,q})-q)(q(B^1_{n,p,q})-p)(q(B^1_{n,p,q})-q-1)x_p \\
&=[q+q(q(B^1_{n,p,q})-q)(q(B^1_{n,p,q})-1)^{n-p-q} \\
&+q(p-1)(q(B^1_{n,p,q})-q-1)(q(B^1_{n,p,q})-1)^{n-p-q}]x_p=0.
\end{align*}
Thus
\begin{align*}
&(q(B^1_{n,p,q})-1)^{n-p-q}[q^3(B^1_{n,p,q})-(p+2q+1)q^2(B^1_{n,p,q}) \\
&+(q^2+pq+p+q)q(B^1_{n,p,q})-q]-q=0
\end{align*}
Let $f(x)=(x-1)^{n-p-q}[x^3-(p+2q+1)x^2+(q^2+pq+p+q)x-q]-q$,
it is not difficult to see that $q(B^1_{n,p,q})$ is the largest
real root of $f(x)=0$. Similarly,
let $g(x)=(x-1)^{n-p-q}[x^3-(q+2p+1)x^2+(p^2+pq+q+p)x-p]-p$, then
$q(B^2_{n,p,q})$ is the largest real root of $g(x)=0$. Since $p>q$,
thus $f(x)-g(x)=(x-1)^{n-p-q}(p-q)[x^2-(p+q)x+1]+p-q>0$, for all $x>p+q$.
Since $q(B^1_{n,p,q})>p+q$ and $q(B^2_{n,p,q})>p+q$, then we have $q(B^1_{n,p,q})<q(B^2_{n,p,q})$.
Therefore, $q(B^1_{n,p,q})\leq q(B^2_{n,p,q})$ with equality if and only if $p=q$.
\end{proof}

\noindent\begin{lemma}\label{le:c6}
 Let $X=(x_1,x_2,\ldots,x_n)^T$ be the
Perron vector corresponding to $q(B^1_{n,p,q})$, where $x_i$ corresponds to the vertex $v_i$,
then we have

(i) $x_1>x_p$;

(ii) $x_{p+1}>x_p$.
\end{lemma}

\begin{proof}

(i) By Theorem \ref{th:ch-1}, we have
\begin{equation}\label{eq:9}
(Q(B^1_{n,p,q})X)_1=q(B^1_{n,p,q})x_1=(q+1)x_1+\sum\limits_{u\in N^+_1}x_u=(q+1)x_1+x_{p+q+1}+qx_{p+1},
\end{equation}
\begin{equation}\label{eq:10}
(Q(B^1_{n,p,q})X)_p=q(B^1_{n,p,q})x_p=qx_p+\sum\limits_{u\in N^+_p}x_u=qx_p+qx_{p+1},
\end{equation}
From (\ref{eq:9}) and (\ref{eq:10}), then we have $$(q(B^1_{n,p,q})-q)(x_1-x_p)=x_1+x_{p+q+1}>0.$$
Thus $x_1>x_p$ by $q(B^1_{n,p,q})>p+q$.

(ii) Since
\begin{equation}\label{eq:11}
(Q(B^1_{n,p,q})X)_{p+1}=q(B^1_{n,p,q})x_{p+1}
 =px_{p+1}+\sum\limits_{u\in N^+_{p+1}}x_u=px_{p+1}+x_1+(p-1)x_p,
 \end{equation}
From (\ref{eq:10}) and (\ref{eq:11}), then we have
$$q(B^1_{n,p,q})(x_{p+1}-x_p)=px_{p+1}-qx_{p+1}+x_1+(p-1-q)x_p.$$
Since $x_1>x_p$, then $$q(B^1_{n,p,q})(x_{p+1}-x_p)>(p-q)(x_{p+1}+x_p)\geq 0.$$
 Thus $x_{p+1}>x_p$ by $q(B^1_{n,p,q})>p+q$.
\end{proof}

\noindent\begin{theorem}\label{th:ch-2} Let $B^3_{n,p,q}=B^1_{n,p,q}-\{(v_n,v_p)\}+\{(v_n,v_1)\}$. Then
$$q(B^3_{n,p,q})>q(B^1_{n,p,q}).$$
\end{theorem}

\begin{proof}
Clearly $B^3_{n,p,q}$ is strongly connected. Let $X=(x_1,x_2,\ldots,x_n)^T$ be the
Perron vector corresponding to $q(B^1_{n,p,q})$, where $x_i$ corresponds to the vertex $v_i$.
Thus $q(B^3_{n,p,q})>q(B^1_{n,p,q})$ by Lemmas \ref{le:c1} and \ref{le:c6}.
\end{proof}

\noindent\begin{theorem}\label{th:ch-3} Let $B^4_{n,p,q}=B^2_{n,p,q}-\{(v_n,v_{p+q})\}+\{(v_n,v_{p+1})\}$. Then
$$q(B^4_{n,p,q})>q(B^2_{n,p,q}).$$
\end{theorem}

\begin{proof}
Clearly $B^4_{n,p,q}$ is strongly connected. Let $X=(x_1,x_2,\ldots,x_n)^T$ be the
Perron vector corresponding to $q(B^2_{n,p,q})$, where $x_i$ corresponds to the vertex $v_i$.
By Lemma \ref{le:c1}, we only need to prove $x_{p+1}>x_{p+q}$.
\begin{equation}\label{eq:12}
(Q(B^2_{n,p,q})X)_{p+1}=q(B^2_{n,p,q})x_{p+1}=(p+1)x_{p+1}+x_{p+q+1}+\sum\limits_{u\in V_p}x_u,
 \end{equation}
\begin{equation}\label{eq:13}
(Q(B^2_{n,p,q})X)_{p+q}=q(B^2_{n,p,q})x_{p+q}=px_{p+q}+\sum\limits_{u\in V_p}x_u,
 \end{equation}
From (\ref{eq:12}) and (\ref{eq:13}), then we have
$$(q(B^2_{n,p,q})-p)(x_{p+1}-x_{p+q})=x_{p+1}+x_{p+q+1}>0$$
Thus $x_{p+1}>x_{p+q}$ by $q(B^2_{n,p,q})>p+q$.
\end{proof}

As the proof of the following result is similar to that of Theorem \ref{th:ch-1}, we omit the details.

\noindent\begin{theorem}\label{th:ch-4} For digraphs $B^5_{n,p,q}$ and $B^6_{n,p,q}$, as shown in Figure \ref{Fig.2},
$$q(B^5_{n,p,q})\leq q(B^6_{n,p,q}),$$ with equality if and only if $p=q$.
\end{theorem}

\noindent\begin{theorem}\label{th:ch-5} For digraphs $B^1_{n,p,q}$ and $B^5_{n,p,q}$, as shown in Figure \ref{Fig.1} and \ref{Fig.2},
$$q(B^1_{n,p,q})< q(B^5_{n-1,p,q}).$$
\end{theorem}

\begin{proof} Since $B^5_{n,p,q}=B^1_{n,p,q}-\{(v_n,v_p)\}+\{(v_n,v_{p+1})\}$
 and $B^5_{n,p,q}$ is strongly connected. Let $X=(x_1,x_2,\ldots,x_n)^T$ be the
Perron vector corresponding to $q(B^1_{n,p,q})$, where $x_i$ corresponds to the vertex $v_i$.
By Lemma \ref{le:c6}, we have $x_{p+1}>x_{p}$.
Thus by Lemma \ref{le:c1}, $q(B^5_{n,p,q})>q(B^1_{n,p,q})$. Then $q(B^5_{n-1,p,q})>q(B^5_{n,p,q})$ by
Lemma \ref{le:c5}. Therefore, we have $q(B^1_{n,p,q})<q(B^5_{n-1,p,q})$.
\end{proof}

\noindent\begin{theorem}\label{th:ch-6} For digraphs $B^1_{n,p,q}$ and $B^5_{n,p,q}$, as shown in Figure \ref{Fig.1} and \ref{Fig.2},
$$q(B^5_{n,p,q})\leq q(B^1_{n-1,p,q}).$$
\end{theorem}

\begin{proof} Let $B^{5*}_{n,p,q}=B^5_{n,p,q}-\{(v_{n-1},v_n)\}+\{(v_{n-1},v_{p})\}$
 and $X=(x_1,x_2,\ldots,x_n)^T$ be the
Perron vector corresponding to $q(B^5_{n,p,q})$, where $x_i$ corresponds to the vertex $v_i$.
Since
$$x_2=x_3=\ldots=x_p\triangleq x_p.$$
$$x_{p+1}=x_{p+2}=\ldots=x_{p+q}\triangleq x_{p+1}.$$
Then from $Q(B^5_{n,p,q})X=q(B^5_{n,p,q})X$, we have
$$
(Q(B^5_{n,p,q})X)_p=q(B^5_{n,p,q})x_p=qx_p+qx_{p+1},
$$
$$
(Q(B^5_{n,p,q})X)_n=q(B^5_{n,p,q})x_n=x_n+x_{p+1},
$$
thus $(q(B^5_{n,p,q})-1)qx_n=qx_{p+1}=(q(B^5_{n,p,q})-q)x_p$.
Since $q(B^5_{n,p,q})-1\geq q(B^5_{n,p,q})-q>0$ by $q(B^5_{n,p,q})>p+q$,
then $x_n\leq x_p$
By Lemma \ref{le:c1}, we have $q(B^{5*}_{n,p,q})\geq q(B^5_{n,p,q})$. Since
$B^1_{n-1,p,q}$ and $K_1=\{v_n\}$ are the two strongly connected components of $B^{5*}_{n,p,q}$,
then by Lemma \ref{le:c2}, we have $q(B^{5*}_{n,p,q})=\max\{q(B^1_{n-1,p,q}), q(K_1)\}=q(B^1_{n-1,p,q})$.
Thus $q(B^5_{n,p,q})\leq q(B^1_{n-1,p,q})$.
\end{proof}

\noindent\begin{theorem}\label{th:ch-7} Let $p\geq q\geq1$, $p+q\leq n-1$, $n\equiv p+q \ (mod \ 2)$ and
$G\in\mathcal{G}_{n,p,q}$ be a bipartite digraph, then $q(G)\geq q(B^5_{n,p,q})$ and the equality holds
if and only if $G\cong B^5_{n,p,q}$.
\end{theorem}

\begin{proof} Clearly, $\overleftrightarrow{K_{p,q}}$ is a proper subdigraph of $G$ since $G\in\mathcal{G}_{n,p,q}$.
Since $G$ is strongly connected, it is possible to obtain a digraph $H$ from $G$ by deleting vertices and arcs in
a way such that one has a subdigraph $\overleftrightarrow{K_{p,q}}$. Therefore

(1) $H\cong B^1_{p+q+k,p,q}$, \ \ \ \ ($k\equiv 1$ \ (mod \ 2), $k\geq1$) or

(2) $H\cong B^2_{p+q+k,p,q}$, \ \ \ \ ($k\equiv 1$ \ (mod \ 2), $k\geq1$) or

(3) $H\cong B^3_{p+q+k,p,q}$, \ \ \ \ ($k\equiv 1$ \ (mod \ 2), $k\geq1$) or

(4) $H\cong B^4_{p+q+k,p,q}$, \ \ \ \ ($k\equiv 1$ \ (mod \ 2), $k\geq1$) or

(5) $H\cong B^5_{p+q+l,p,q}$, \ \ \ \ ($l\equiv 0$ \ (mod \ 2), $l\geq2$) or

(6) $H\cong B^6_{p+q+l,p,q}$, \ \ \ \ ($l\equiv 0$ \ (mod \ 2), $l\geq2$).

By Lemma \ref{le:c3}, $q(H)\leq q(G)$, the equality holds if and only if $H\cong G$.

Case (i). $H\cong B^1_{p+q+k,p,q}$, \ \ \ \ ($k\equiv 1$ \ (mod \ 2), $k\geq1$)

Insert $n-p-q-k-1$ vertices into the directed $\overrightarrow{P_{k+2}}$ such that
the resulting bipartite digraphs is $B^1_{n-1,p,q}$, then
$q(B^1_{n-1,p,q})\leq q(H)$ by using Lemma \ref{le:c5} repeatedly $n-p-q-k-1$ times,
and thus $q(B^5_{n,p,q})\leq q(B^1_{n-1,p,q})\leq q(H)<q(G)$ by Theorem \ref{th:ch-6}.

Case (ii). $H\cong B^2_{p+q+k,p,q}$, \ \ \ \ ($k\equiv 1$ \ (mod \ 2), $k\geq1$)

Insert $n-p-q-k-1$ vertices into the directed $\overrightarrow{P_{k+2}}$ such that
the resulting bipartite digraphs is $B^2_{n-1,p,q}$, then
$q(B^2_{n-1,p,q})\leq q(H)$ by using Lemma \ref{le:c5} repeatedly $n-p-q-k-1$ times,
and thus $ q(B^5_{n,p,q})\leq q(B^1_{n-1,p,q})\leq q(B^2_{n-1,p,q})\leq q(H)<q(G)$
by Theorems \ref{th:ch-1} and \ref{th:ch-6}.

Case (iii). $H\cong B^3_{p+q+k,p,q}$, \ \ \ \ ($k\equiv 1$ \ (mod \ 2), $k\geq1$)

Insert $n-p-q-k-1$ vertices into the directed $\overrightarrow{C_{k+1}}$ such that
the resulting bipartite digraphs is $B^3_{n-1,p,q}$, then
$q(B^3_{n-1,p,q})\leq q(H)$ by using Lemma \ref{le:c5} repeatedly $n-p-q-k-1$ times,
and thus  $q(B^5_{n,p,q})\leq q(B^1_{n-1,p,q})< q(B^3_{n-1,p,q})\leq q(H)<q(G)$
by Theorems \ref{th:ch-2} and \ref{th:ch-6}.

Case (iv). $H\cong B^4_{p+q+k,p,q}$, \ \ \ \ ($k\equiv 1$ \ (mod \ 2), $k\geq1$)

Insert $n-p-q-k-1$ vertices into the directed $\overrightarrow{C_{k+1}}$ such that
the resulting bipartite digraphs is $B^4_{n-1,p,q}$, then
$q(B^4_{n-1,p,q})\leq q(H)$ by using Lemma \ref{le:c5} repeatedly $n-p-q-k-1$ times,
and thus $q(B^5_{n,p,q})\leq q(B^1_{n-1,p,q})\leq q(B^2_{n-1,p,q})\leq q(B^4_{n-1,p,q})\leq q(H)<q(G)$
by Theorems \ref{th:ch-1}, \ref{th:ch-3} and \ref{th:ch-6}.

Case (v). $H\cong B^5_{p+q+l,p,q}$, \ \ \ \ ($l\equiv 0$ \ (mod \ 2), $l\geq2$)

Insert $n-p-q-l$ vertices into the directed $\overrightarrow{P_{l+2}}$ such that
the resulting bipartite digraphs is $B^5_{n,p,q}$, then
$q(B^5_{n,p,q})\leq q(H)$ by using Lemma \ref{le:c5} repeatedly $n-p-q-l$ times,
and thus $q(B^5_{n,p,q})\leq q(H)\leq q(G)$.

Case (vi). $H\cong B^6_{p+q+l,p,q}$, \ \ \ \ ($l\equiv 0$ \ (mod \ 2), $l\geq2$)

Insert $n-p-q-l$ vertices into the directed $\overrightarrow{P_{l+2}}$ such that
the resulting bipartite digraphs is $B^6_{n,p,q}$, then
$q(B^6_{n,p,q})\leq q(H)$ by using Lemma \ref{le:c5} repeatedly $n-p-q-l$ times,
and thus $q(B^5_{n,p,q})\leq q(B^6_{n,p,q})\leq q(H)\leq q(G)$ by Theorem \ref{th:ch-4}.

Combining the above six cases, we have $q(G)\geq q(B^5_{n,p,q})$ and the equality holds
if and only if $G\cong B^5_{n,p,q}$, where $n\equiv p+q \ (mod \ 2)$.
\end{proof}

\noindent\begin{theorem}\label{th:ch-8} Let $p\geq q\geq1$, $p+q\leq n-1$, $n\not\equiv p+q \ (mod \ 2)$ and
$G\in\mathcal{G}_{n,p,q}$ be a bipartite digraph, then $q(G)\geq q(B^1_{n,p,q})$ and the equality holds
if and only if $G\cong B^1_{n,p,q}$.
\end{theorem}

\begin{proof} Clearly, $\overleftrightarrow{K_{p,q}}$ is a proper subdigraph of $G$ since $G\in\mathcal{G}_{n,p,q}$.
Since $G$ is strongly connected, it is possible to obtain a digraph $H$ from $G$ by deleting vertices and arcs in
a way such that one has a subdigraph $\overleftrightarrow{K_{p,q}}$. Therefore

(1) $H\cong B^1_{p+q+k,p,q}$, \ \ \ \ ($k\equiv 1$ \ (mod \ 2), $k\geq1$) or

(2) $H\cong B^2_{p+q+k,p,q}$, \ \ \ \ ($k\equiv 1$ \ (mod \ 2), $k\geq1$) or

(3) $H\cong B^3_{p+q+k,p,q}$, \ \ \ \ ($k\equiv 1$ \ (mod \ 2), $k\geq1$) or

(4) $H\cong B^4_{p+q+k,p,q}$, \ \ \ \ ($k\equiv 1$ \ (mod \ 2), $k\geq1$) or

(5) $H\cong B^5_{p+q+l,p,q}$, \ \ \ \ ($l\equiv 0$ \ (mod \ 2), $l\geq2$) or

(6) $H\cong B^6_{p+q+l,p,q}$, \ \ \ \ ($l\equiv 0$ \ (mod \ 2), $l\geq2$).

By Lemma \ref{le:c3}, $q(H)\leq q(G)$, the equality holds if and only if $H\cong G$.

Case (i). $H\cong B^1_{p+q+k,p,q}$, \ \ \ \ ($k\equiv 1$ \ (mod \ 2), $k\geq1$)

Insert $n-p-q-k$ vertices into the directed $\overrightarrow{P_{k+2}}$ such that
the resulting bipartite digraphs is $B^1_{n,p,q}$, then
$q(B^1_{n,p,q})\leq q(H)$ by using Lemma \ref{le:c5} repeatedly $n-p-q-k$ times,
and thus $q(B^1_{n,p,q})\leq q(H)\leq q(G)$.

Case (ii). $H\cong B^2_{p+q+k,p,q}$, \ \ \ \ ($k\equiv 1$ \ (mod \ 2), $k\geq1$)

Insert $n-p-q-k$ vertices into the directed $\overrightarrow{P_{k+2}}$ such that
the resulting bipartite digraphs is $B^2_{n,p,q}$, then
$q(B^2_{n,p,q})\leq q(H)$ by using Lemma \ref{le:c5} repeatedly $n-p-q-k$ times,
and thus $q(B^1_{n,p,q})\leq q(B^2_{n,p,q})\leq q(H)\leq q(G)$
by Theorem \ref{th:ch-1}.

Case (iii). $H\cong B^3_{p+q+k,p,q}$, \ \ \ \ ($k\equiv 1$ \ (mod \ 2), $k\geq1$)

Insert $n-p-q-k$ vertices into the directed $\overrightarrow{C_{k+1}}$ such that
the resulting bipartite digraphs is $B^3_{n,p,q}$, then
$q(B^3_{n,p,q})\leq q(H)$ by using Lemma \ref{le:c5} repeatedly $n-p-q-k$ times,
and thus  $q(B^1_{n,p,q})< q(B^3_{n,p,q})\leq q(H)\leq q(G)$
by Theorem \ref{th:ch-2}.

Case (iv). $H\cong B^4_{p+q+k,p,q}$, \ \ \ \ ($k\equiv 1$ \ (mod \ 2), $k\geq1$)

Insert $n-p-q-k$ vertices into the directed $\overrightarrow{C_{k+1}}$ such that
the resulting bipartite digraphs is $B^4_{n-1,p,q}$, then
$q(B^4_{n,p,q})\leq q(H)$ by using Lemma \ref{le:c5} repeatedly $n-p-q-k$ times,
and thus $q(B^1_{n,p,q})\leq q(B^2_{n,p,q})<q(B^4_{n,p,q})\leq q(H)\leq q(G)$
by Theorems \ref{th:ch-1} and \ref{th:ch-3}.

Case (v). $H\cong B^5_{p+q+l,p,q}$, \ \ \ \ ($l\equiv 0$ \ (mod \ 2), $l\geq2$)

Insert $n-p-q-l-1$ vertices into the directed $\overrightarrow{P_{l+2}}$ such that
the resulting bipartite digraphs is $B^5_{n-1,p,q}$, then
$q(B^5_{n-1,p,q})\leq q(H)$ by using Lemma \ref{le:c5} repeatedly $n-p-q-l-1$ times,
and thus $q(B^1_{n,p,q})<q(B^5_{n-1,p,q})\leq q(H)< q(G)$
by Theorem \ref{th:ch-5}.

Case (vi). $H\cong B^6_{p+q+l,p,q}$, \ \ \ \ ($l\equiv 0$ \ (mod \ 2), $l\geq2$)

Insert $n-p-q-l-1$ vertices into the directed $\overrightarrow{P_{l+2}}$ such that
the resulting bipartite digraphs is $B^6_{n-1,p,q}$, then
$q(B^6_{n-1,p,q})\leq q(H)$ by using Lemma \ref{le:c5} repeatedly $n-p-q-l-1$ times,
and thus $q(B^1_{n,p,q})<q(B^5_{n-1,p,q})\leq q(B^6_{n-1,p,q})\leq q(H)< q(G)$
by Theorems \ref{th:ch-4} and \ref{th:ch-5}.

Combining the above six cases, we have $q(G)\geq q(B^1_{n,p,q})$ and the equality holds
if and only if $G\cong B^1_{n,p,q}$, where $n\not\equiv p+q \ (mod \ 2)$.
\end{proof}

%\section*{Acknowledgements}

%The authors are very grateful to the referees for many detailed
%comments and suggestions, which are very helpful for improving the
%presentation of the manuscript. Lemma 2.8 and Theorem 3.9 are due
%to an anonymous referee, to whom we would like to express our
%special thanks.

\end{document}